\documentclass[11pt,fleqn]{article}
\usepackage{amssymb}
\usepackage{amsfonts}
\usepackage{amsmath}
\usepackage{amsthm}
\usepackage{graphicx}
\usepackage{epsfig}
\usepackage{psfrag}
\usepackage{color}
\usepackage{dsfont}
\makeatletter
\xdef\@endgadget#1{{\unskip\nobreak\hfil\penalty50\hskip1em\hbox{}\nobreak
    \hfil#1\parfillskip=0pt\finalhyphendemerits=0\par}}
\def\@qedsymbol{${}_\blacksquare$}
\def\qed{\@endgadget{\@qedsymbol}}
\newtheorem{lemma}{Lemma}[section]
\newtheorem{theorem}[lemma]{Theorem}

\newtheorem{example}[lemma]{Example}
\newtheorem{definition}[lemma]{Definition}

\newtheorem{proposition}[lemma]{Proposition}
\newtheorem{remark}[lemma]{Remark}
\newtheorem{assumption}[lemma]{Assumption}
\newcommand{\mR}{\mathbb{R}}

\newcommand{\bq}{\begin{equation}}
\newcommand{\eq}{\end{equation}}
\DeclareMathOperator{\im}{im}

\DeclareMathOperator{\spa}{span}
\DeclareMathOperator{\sign}{sign}
\newcommand{\bma}{\begin{bmatrix}}
\newcommand{\ema}{\end{bmatrix}}

\def\BibTeX{{\rm B\kern-.05em{\sc i\kern-.025em b}\kern-.08em
    T\kern-.1667em\lower.7ex\hbox{E}\kern-.125emX}}

\title{\LARGE \bf Kron Reduction of Nonlinear Networks}

\author{Arjan van der Schaft, Bart Besselink, Anne-Men Huijzer
\thanks{Arjan van der Schaft, Bart Besselink and Anne-Men Huijzer are all with the Bernoulli Institute for Mathematics, Computer
Science and AI, Jan C. Willems Center for Systems and Control, University of Groningen, PO Box 407, 9700 AK, the
Netherlands
        {\tt\small A.J.van.der.Schaft@rug.nl, b.besselink@rug.nl, m.a.huijzer@rug.nl}}
}

\begin{document}

\maketitle
\thispagestyle{empty}
\pagestyle{empty}

\begin{abstract}
Kron reduction is concerned with the elimination of interior nodes of physical network systems, such as linear resistive circuits. In this paper it is shown how this can be extended to networks with \emph{nonlinear} static relations between the variables associated to the edges of the underlying graph.
\end{abstract}
%

\section{Introduction}
\label{sec1}
In many applications of network theory the following scenario arises. Consider a graph, with pairs of variables associated to its nodes and to its edges. Suppose the set of nodes is split into a set of 'terminals' (called 'boundary' nodes in the present paper) and a complementary set of 'interior' nodes ('central' nodes in this paper). Suppose that one is primarily interested in the relation between the variables at the terminals of the network. Then it is advantageous if one can eliminate the central nodes of the graph, so as to obtain a reduced graph whose node set only consists of the boundary nodes, and which is equivalent to the original network from an input-output point of view (where the inputs and outputs are the variables at the boundary nodes). Such an elimination process is typically referred to as \emph{Kron reduction} \cite{kron}. A related problem was originally studied by Kirchhoff \cite{kirchhoff} in 1847. He considered electrical circuits consisting of linear resistors satisfying Ohm's law $V=RI$, with $V$ and $I$ denoting the voltage, respectively, current, and $R$ the resistance. Furthermore, in nowadays terminology, the voltage potentials, respectively, nodal currents, at the boundary nodes are considered to be the inputs and outputs of the circuit. It is assumed that the nodal currents at the central nodes are equal to zero (no in/outgoing external currents), and that the currents through the edges add up to zero at every central node ('Kirchhoff's current law'). Also, the voltages across edges are given as differences of voltage potentials at the head and tail nodes of each edge ('Kirchhoff's voltage law'). Loosely speaking, the question posed (and answered affirmatively) by Kirchhoff was if for every value of the input vector (the voltage potentials at the boundary nodes) there exists a unique vector of voltage potentials at the central nodes, thus determining a unique solution of the currents through, and voltages across, all the resistors of the circuit, and unique nodal currents at the boundary nodes (the vector of outputs).

Other examples of Kron reduction appear in various distribution networks. Also, there is a close relation with network flow theory and network optimization, see e.g. \cite{rockafellar}. Furthermore, instead of static networks such as resistor circuits, one may also consider \emph{dynamic} networks. In fact, Kron reduction is especially known from power networks; see e.g. \cite{doerfler, caliskan} and the references quoted therein. See also \cite{bakker} for Kron reduction of chemical reaction networks.

In the present paper we will consider the general problem of Kron reduction of networks, where the static relations between the variables associated to the edges are not anymore \emph{linear} (like in Ohm's law), but instead \emph{nonlinear}. More precisely, we will assume that the relation between the variables at each edge is given by a differentiable and monotone conductance function. In this way, the problem studied in this paper is related to the theory of monotone networks, see e.g. \cite{minty}, as well as of convex optimization, see e.g. \cite{rockafellar}. However, the approach of the present paper remains within a \emph{differentiable} setting, relying on basic tools from analysis (such as the implicit function theorem), algebraic graph theory (incidence and Laplacian matrices), and linear algebra (Schur complements).

The structure of the paper is as follows. Section \ref{sec2} starts with a very brief recall of Kron reduction of linear static networks, and then proceeds towards the main results for Kron reduction of static nonlinear networks (Theorems \ref{theoremone} and \ref{theoremtwo}). 
%
In Section \ref{sec3} two applications of the main results will be discussed. The first concerns Kron reduction of electrical circuits consisting of nonlinear resistors. A special case occurs if only two boundary nodes are selected. In this case, Kron reduction yields the effective nonlinear conductance between these two nodes. This will be illustrated by a nonlinear electrical circuit example. The second application area concerns memristor networks. Ideal memristors are intrinsically nonlinear elements (if they are linear, then they reduce to ordinary resistors). Hence nonlinear Kron reduction and the computation of the effective memductance provide key tools for the analysis of memristor networks. Finally 
Section \ref{sec5} points to open problems for further research.

\begin{remark}
{\rm
While finishing this paper we came across the recent paper \cite{klimm}, addressing Kron reduction of nonlinear networks from a network optimization perspective. A main result stated in this paper is that Kron reduction of nonlinear networks is often \emph{not} possible. However, it should be emphasized that the relation for the flow $f_j$ through the $j$-th edge is assumed to be of the form $\sign (f_j) \beta_j |f_j|^r = z_h - z_t$, with $z_h,z_t$ the potentials at the head, respectively,  tail node of this edge. For $r=2$ (the main case investigated in \cite{klimm}) this is not a smooth relation as considered in the present paper.
}
\end{remark}

\section{Nonlinear Kron reduction}
\label{sec2}
\subsection{Brief recall of the linear case}
Consider a connected directed graph, with $n$ nodes and $m$ edges, specified by its $n \times m$ incidence matrix $D$. Let the $i$-th component of $z \in \mR^n$ be a real-valued variable associated with the $i$-th node, $i=1,\cdots,n$, and the $j$-th component of $y \in \mR^m$ be a real-valued variable associated to the $j$-th edge, $j=1,\cdots,m$. Impose the relation $y=D^\top z$ ('Kirchhoff's voltage law').

Furthermore, consider an $m \times m$ positive diagonal matrix $\bar{G}$, whose diagonal elements $\bar{g}_j$ specify linear functions $\bar{g}_jy_j$ on the edges. Then consider the Laplacian matrix $L:= D\bar{G}D^\top$; see Theorem \ref{theorem1}. Split the nodes into central nodes ('$C$') and boundary nodes ('$B$'), and correspondingly split $z=(z_C,z_B)$. Then consider the constraints $D_C\bar{G}D^\top z=0$ ('nodal currents at central nodes are zero'), where $D_C$ denotes the upper block of $D$ corresponding to the central nodes. Solving these constraints corresponds to taking the Schur complement of the Laplacian matrix $L=D\bar{G}D^\top$ with respect to its $(C,C)$ (left-upper) block. The crucial observation is that such a Schur complement is again a Laplacian matrix, corresponding to a graph with node set given by the boundary nodes of the original graph, and with a new set of edges in between these nodes.
To the best of our knowledge this was first proven in \cite{vds10}, making essential use of an argument in \cite{willemsverriest}.

\subsection{Nonlinear networks}
For the purpose of this paper a \emph{nonlinear network} is defined as follows. As in the linear case, consider a connected directed graph, specified by an $n \times m$ incidence matrix $D$, and vectors $z \in \mR^n$ and $y \in \mR^m$ such that $y=D^\top z$. However, differently from the linear case, consider for each $j$-th edge a scalar twice differentiable \emph{strongly convex} function $G_j(y_j)$, leading to strictly monotone 'conductance' functions $g_j(y_j)= \frac{\partial G_j}{\partial y_j}(y_j), j=1, \cdots, m$. (Conversely, one can start with strictly monotone functions $g_j(y_j)$, and integrate them to strongly convex functions $G_j(y_j)$.) Summarizing:
\begin{definition}[Nonlinear network]
\label{nonnet}
A \emph{nonlinear network} is defined by a connected directed graph with $n \times m$ incidence matrix $D$, where the $j$-th edge is endowed with a twice differentiable strongly convex function $G_j(y_j), j=1, \cdots,m$.
\end{definition}
\begin{remark}
{\rm
A more abstract and intrinsic definition of the vectors $z \in \mR^n$ and $y\in \mR^m$ corresponding to the nodes, respectively, edges of the graph can be given as follows. Consider the space of functions from the node set of the graph to $\mR$; denoted by $C_0$ (the {\it node space}). Obviously, this is a linear space, which can be identified with $\mR^n$. Denote the dual space by $C^0$. Then $z \in C^0$. Analogously, consider the space of functions from the edge set of the graph to $\mR$; denoted by $C_1$ (the {\it edge space}). This is a linear space that can be identified with $\mR^m$. Denote the dual space by $C^1$; then $y \in C^1$. 
}
\end{remark}

Kron reduction is now defined as follows. Consider a nonlinear network, and define the twice differentiable function $G:\mR^m \to \mR$ as $G(y):=\sum_{j=1}^m G_j(y_j)$, which obviously is strongly convex as well. Subsequently, define the function $K: \mR^n \to \mR$ as
\bq
\label{K}
K(z):= G(D^\top z) .
\eq
Also $K$ is convex, as well as twice differentiable. Furthermore, since $\mathds{1} \in \ker D^\top$ the function $K$ is invariant under \emph{shifts} $z \mapsto z + c \mathds{1}$ for any $c \in \mR$ (implying that $K$ is \emph{not} strictly convex). 
One computes 
\bq
\label{K1}
\frac{\partial K}{\partial z}(z) = D \frac{\partial G}{\partial y}(D^\top z), \quad \frac{\partial^2 K}{\partial z^2}(z) = D \frac{\partial^2 G}{\partial y^2}(D^\top z) D^\top .
\eq
In particular, $\frac{\partial^2 K}{\partial z^2}(z)$ is for every $z$ a \emph{Laplacian matrix} with weight matrix $\frac{\partial^2 G}{\partial y^2}(D^\top z)$; see the Appendix for background on Laplacian matrices. Indeed, $\frac{\partial^2 G}{\partial y^2}(D^\top z)$ is \emph{diagonal}, since $G(y):=\sum_{j=1}^m G_j(y_j)$, and its diagonal elements are \emph{positive}, since the functions $G_j$ are strongly convex.
\begin{remark}
\label{resis}
{\rm In the context of a \emph{nonlinear resistor network}, to be discussed in more detail in Section \ref{sec3}, $z \in C^0$ is the vector of \emph{voltage potentials} associated with the nodes, and $y=D^\top z \in C^1$ is the vector of \emph{voltages} across the edges. The functions $G_j$ define the strictly monotone conductance functions $I_j=\frac{\partial G_j}{\partial y_j}(y_j) \in C_1$, with $I_j$ the \emph{current} through the $j$-th edge. Furthermore, $\frac{\partial K}{\partial z}(z) \in C_0$ is the vector of \emph{nodal} currents at the nodes.
}
\end{remark}

Consider now a \emph{splitting of the node set} into a set of 'boundary' nodes, and the complementary set of interior or 'central' nodes; the last ones to be eliminated by Kron reduction. Correspondingly, split the vector $z$ into $z=(z_C,z_B)$, with $B$ referring to the boundary nodes, and $C$ to the remaining central nodes. Now, consider the \emph{constraints}
\bq
\label{constraint}
0 = \frac{\partial K}{\partial z_C}(z_C,z_B) .
\eq
(In the nonlinear resistor network interpretation this amounts to setting the nodal currents at the central nodes equal to zero.)
By strong convexity of $G_j, j=1, \cdots,m$, and the fact that the graph is assumed to be connected, it follows from \eqref{K1} (see Theorem \ref{theorem1}) that
\bq
\frac{\partial^2 K}{\partial z^2_C}(z_C,z_B) >0 .
\eq
Hence by the Implicit Function theorem, if there exists $\bar{z}_C,\bar{z}_B$ satisfying \eqref{constraint}, then locally around $\bar{z}_C,\bar{z}_B$ one may solve the constraint equations \eqref{constraint} uniquely for $z_C$, and express $z_C$ as a smooth function $z_C(z_B)$ of $z_B$. A standing assumption throughout this paper is that this solvability is \emph{global}.
\begin{assumption}
\label{ass1}
For each $z_B$ the constraint equations \eqref{constraint} can be solved globally for $z_C$, resulting in a differentiable mapping $z_C(z_B)$ satisfying 
\bq
\label{constraintsolved}
0 = \frac{\partial K}{\partial z_C}(z_C(z_B),z_B), \quad \mbox{ for all } z_B .
\eq
\end{assumption}

\smallskip

By differentiating \eqref{constraintsolved} with respect to $z_B$ we obtain
\bq
0 = \frac{\partial^2 K}{\partial z^2_C}(z_C(z_B),z_B) \frac{\partial z_C}{\partial z_B}(z_B) \, + 
 \frac{\partial^2 K}{\partial z_C \partial z_B}(z_C(z_B),z_B) .
\eq
This implies
\bq
\label{zCzB}
\frac{\partial z_C}{\partial z_B}(z_B) \!= 
\!- \! \left[\frac{\partial^2 K}{\partial z^2_C}(z_C(z_B),z_B)\right]^{-1} \! \! \!\frac{\partial^2 K}{\partial z_B \partial z_C}(z_C(z_B),z_B).
\eq
Now define the function
\bq
\widehat{K}(z_B) := K(z_C(z_B),z_B) .
\eq
It follows from Lemma \ref{lemmaA} that $\widehat{K}$ is convex, and from Lemma \ref{lemmaB} that $\widehat{K}$ is invariant under shifts $z_B \mapsto z_B + c \mathds{1}$.
Furthermore
\bq
\label{appe}
\begin{array}{l}
\frac{\partial \widehat{K}}{\partial z_B}(z_B)  =  \frac{\partial K}{\partial z_C}(z_C(z_B),z_B) \frac{\partial z_C}{\partial z_B}(z_B) +  \frac{\partial K}{\partial z_B}(z_C(z_B),z_B)\\[3mm]
 \qquad \qquad = \frac{\partial K}{\partial z_B}(z_C(z_B),z_B),
 \end{array}
\eq
since $0=\frac{\partial K}{\partial z_C}(z_C(z_B),z_B)$. Hence
\bq
\label{hat}
\frac{\partial^2 \widehat{K}}{\partial z^2_B}(z_B) = 
\quad \frac{\partial^2 K}{\partial z^2_B}(z_C(z_B),z_B) + \frac{\partial^2 K}{\partial z_B \partial z_C}(z_C(z_B),z_B) \frac{\partial z_C}{\partial z_B}(z_B)  .
\eq
By using \eqref{zCzB}, equation \eqref{hat} implies
\bq
\label{Schur}
\begin{array}{l}
\frac{\partial^2 \widehat{K}}{\partial z^2_B}(z_B) = \frac{\partial^2 K}{\partial z^2_B}(z_C(z_B),z_B) \, -\\
\frac{\partial^2 K}{\partial z_B \partial z_C}(z_C(z_B),z_B) \cdot \left(\frac{\partial^2 K}{\partial z^2_C}(z_C(z_B),z_B)\right)^{-1} \cdot 
\frac{\partial^2 K}{\partial z_C \partial z_B}(z_C(z_B),z_B) .
\end{array}
\eq
For fixed $z_B$ this can be recognized as the \emph{Schur complement} of the Laplacian matrix $\frac{\partial^2 K}{\partial z^2}(z_C(z_B),z_B)$ with respect to its sub-block $\frac{\partial^2 K}{\partial z^2_C}(z_C(z_B),z_B)$. 

Since any Schur complement of a Laplacian matrix is again a Laplacian matrix, cf. Theorem 3.1 in \cite{vds10} as summarized in Appendix Theorem \ref{theorem1}, it follows that for any fixed $z_B$ there exists an incidence matrix $\widehat{D}$ of a \emph{reduced graph}, called the \emph{Kron reduced graph}, with the \emph{same} boundary nodes but \emph{without} central nodes, such that
\bq
\label{Laplacehat}
\frac{\partial^2 \widehat{K}}{\partial z^2_B}(z_B) = \widehat{D} W(z_B) \widehat{D}^\top,
\eq
where $W(z_B)$ is a diagonal matrix with positive elements. In fact, this reduced graph has an edge between node $i$ and $k$ if and only if the $(i,k)$-th element of the Schur complement \eqref{Schur} is different from zero. Hence for any small variation of $z_B$ edges will remain edges, while new edges may appear when zero elements in \eqref{Schur} become non-zero. This motivates to make throughout this paper the following assumption.
\begin{assumption}
\label{ass2}
The Kron reduced graph has the {\it same} incidence matrix $\widehat{D}$ (with $\widehat{m}$ edges) for any $z_B$.
\end{assumption}

As summarized in Theorem \ref{theorem1}, since the original graph is assumed to be connected, also the Kron reduced graph is connected. Hence $\ker \widehat{D}^\top = \spa \mathds{1}$. Since $\widehat{K}$ is invariant under shifts $z_B \mapsto z_B + c \mathds{1}$, it thus follows that $W(z_B)$ only depends on $\widehat{D}^\top z_B$. Hence, one can define the diagonal matrix $\widehat{W}(\hat{y}), \hat{y} \in \mR^{\hat{m}}$, such that
\bq
W(z_B)= \widehat{W}(\widehat{D}^\top z_B) .
\eq
Note that $\hat{y} \in \mR^{\hat{m}}$ is a vector of variables corresponding to the $\hat{m}$ edges of the Kron reduced graph.

Furthermore, since $\widehat{K}$ is invariant under shifts $z_B \mapsto z_B + c \mathds{1}$, there exists a twice differentiable function $\widehat{G}(\hat{y})$ such that 
\bq
\label{fact}
\widehat{K}(z_B) = \widehat{G} (\widehat{D}^\top z_B) .
\eq
By direct computation one obtains
\bq
\label{hessK}
\frac{\partial^2 \widehat{K}}{\partial z^2_B}(z_B) = \widehat{D} \frac{\partial^2 \widehat{G}}{\partial \hat{y}^2}(\widehat{D}^\top z_B) \widehat{D}^\top .
\eq
Now compare the expressions \eqref{hessK} and \eqref{Laplacehat}. First note the following simple linear-algebraic lemma.
\begin{lemma}
\label{lemma}
Let $B$ be a matrix, and let $F$ be a matrix such that $\im F=\ker B$. Then $BAB^\top =0$ for a symmetric matrix $A$ if and only if $A$ is of the form $A=FSF^\top$ for a symmetric matrix $S$. In particular, if $\ker B=0$ then $A=0$.
\end{lemma}

Let us now first proceed under the simplifying assumption $\ker \widehat{D}=0$; i.e., the Kron reduced graph does not contain cycles. Equivalently, assume that $\im \widehat{D}^\top= \mR^{\hat{m}}$.
Then we obtain the following proposition.
\begin{proposition}
If $\ker \widehat{D}= \{0\}$ then $\widehat{G}$ satisfying \eqref{fact} is uniquely determined. Furthermore, $\widehat{W}$ is given as
\bq
\label{diagonal}
\widehat{W}(\hat{y})= \frac{\partial^2 \widehat{G}}{\partial \hat{y}^2}(\hat{y}) .
\eq
\end{proposition}
\begin{proof}
Unicity of $\widehat{G}: \mR^{\hat{m}} \to \mR$ directly follows from $\im \widehat{D}^\top= \mR^{\hat{m}}$. Lemma \ref{lemma} yields \eqref{diagonal}.
\end{proof}

This, in turn, leads to the following proposition.
\begin{proposition}
\label{decomposition}
By \eqref{diagonal} and the fact that $\widehat{W}$ is diagonal
\bq
\frac{\partial^2 \widehat{G}}{\partial \hat{y}_i \partial \hat{y}_k}(\hat{y})=0
\eq
for $i \neq k$. This implies that there exist twice differentiable functions $\widehat{G}_j(\hat{y}_j), j=1, \cdots,\hat{m},$ such that
\bq
\label{sum}
\widehat{G}(\hat{y}) = \widehat{G}_1(\hat{y}_1) + \cdots + \widehat{G}_{\hat{m}}(\hat{y}_{\hat{m}}) .
\eq
Furthermore, since the diagonal elements of $W(z_B)$ are all positive, $\widehat{G}_j, j=1, \cdots,\hat{m}$, are strongly convex functions.
\end{proposition}
\begin{proof}
Since $\frac{\partial^2 \widehat{G}}{\partial \hat{y}_i \partial \hat{y}_k}(\hat{y})=0$ for all $i \neq k$ it follows that $\frac{\partial \widehat{G}}{\partial \hat{y}_j}(\hat{y})$ only depends on $\hat{y}_j$. Thus
\bq
\frac{\partial \widehat{G}}{\partial \hat{y}_j}(\hat{y}) = \widehat{g}_j(\hat{y}_j), \quad j=1, \cdots \hat{m}
\eq
for some scalar functions $\widehat{g}_j$. Then define
\bq
\widehat{G}_j(\hat{y}_j) := \int_0^{\hat{y}_j} \widehat{g}_j(v_j) d v_j, \quad j=1, \cdots \hat{m}
\eq
It follows that
\bq
\frac{\partial}{\partial \hat{y}_j}\left[ \widehat{G} - (\widehat{G}_1 + \cdots \widehat{G}_{\hat{m}}) \right](\hat{y}) =0, \quad j=1, \cdots \hat{m},
\eq
and thus \eqref{sum} holds up to a constant (which can be incorporated in any of the functions $\widehat{G}_j$).
\end{proof}

This is summarized as follows.
\begin{theorem}
\label{theoremone}
Consider a nonlinear network with nodes split into boundary and central nodes. Let Assumptions \ref{ass1} and \ref{ass2} be satisfied. Then the constraint equations \eqref{constraint}, where $K$ is defined by \eqref{K}, lead to a Kron reduced graph with incidence matrix $\widehat{D}$ and $\hat{m}$ edges. Assume that $\ker \widehat{D}=0$ (i.e., the reduced graph does not have cycles), then there exist twice differentiable strongly convex functions $\widehat{G}_j,  j=1, \cdots, \hat{m},$ resulting in a reduced nonlinear network as in Definition \ref{nonnet}.
\end{theorem}

Next consider the \emph{general} case where $\ker \widehat{D}$ need not be equal to $0$, i.e., the Kron reduced graph may have cycles. In this case, \eqref{diagonal} need not hold for $\widehat{G}$ satisfying \eqref{fact}. On the other hand, in this case $\widehat{G}: \mR^{\hat{m}} \to \mR$ is uniquely determined \emph{only} on the subspace $\im \widehat{D}^\top$, and we may exploit this freedom in choosing $\widehat{G}$ such that \eqref{diagonal} becomes satisfied.

First note that by application of Lemma \ref{lemma} we obtain in the general case, instead of \eqref{diagonal}, the extended identity
\bq
\label{diagonal1}
\widehat{W}(\hat{y}) = \frac{\partial^2 \widehat{G}}{\partial \hat{y}^2}(\hat{y}) + FS(\hat{y})F^\top,
\eq
for some symmetric matrix $S(\hat{y})$, where $\im F= \ker \widehat{D}$. 

Denote the number of boundary nodes by $n_B$. Since $\ker \widehat{D}^\top = \spa \mathds{1}$ one can perform linear coordinate changes of $\hat{z}$ and $\hat{y}$ such that the $n_B \times \hat{m}$ matrix $\widehat{D}$ and the matrix $F$ in the new coordinates take the form
\bq
\widehat{D} = \bma I_\ell & 0 \\[2mm] 0 & 0 \ema, \qquad F= \bma 0 \\[2mm] I_{\hat{m} - \ell} \ema,
\eq
where $\ell = n_B -1$. In such coordinates $\im \widehat{D}^\top = \im \bma I_\ell \\ 0 \ema$. Since $\widehat{G}: \mR^{\hat{m}} \to \mR$  is fixed by \eqref{fact} \emph{only} on the subspace $\im \widehat{D}^\top \subset \mR^{\widehat{m}}$, it can be chosen \emph{freely} on the complementary subspace $\im \bma 0_\ell \\ I_{\hat{m} - \ell} \ema$. Using the coordinate expression for $\im F$ it follows that $\widehat{G}$ can be (re-)defined in such a way that the last term in \eqref{diagonal1} is compensated, \emph{provided} the matrix $S(\hat{y})$ is a Hessian matrix. This, in turn, holds if and only if the elements $S_{ij}(\hat{y})$ of the matrix $S(\hat{y})$ satisfy the following integrability conditions
\bq
\label{integ}
\frac{\partial S_{ij}}{\partial \hat{y}_k}(\hat{y}) = \frac{\partial S_{kj}}{\partial \hat{y}_i}(\hat{y}), \quad i,j,k=1, \cdots, \hat{m},
\eq
see e.g. \cite{vds11} for a proof.
Thus we obtain the following theorem in the general case.
\begin{theorem}
\label{theoremtwo}
Consider a nonlinear network, with the nodes split into boundary and central nodes. Let Assumptions \ref{ass1} and \ref{ass2} be satisfied. Then the constraint equations \eqref{constraint}, with $K$ defined by \eqref{K}, lead to Kron reduced graph with incidence matrix $\widehat{D}$ and $\hat{m}$ edges. Assume additionally that the matrix $S(\hat{y})$ in \eqref{diagonal1} is a Hessian matrix, or, equivalently, its elements satisfy \eqref{integ}. Then there exist twice differentiable strongly convex functions $\widehat{G}_j,  j=1, \cdots, \hat{m},$ defining a nonlinear network as in Definition \ref{nonnet}.
\end{theorem}
\begin{remark}
{\rm
Note that all results on the relation between the original graph and the Kron reduced graph in the linear case remain true in the nonlinear setting, since the Kron reduced graph is still determined by the Schur complement \eqref{Schur} of the Laplacian matrix corresponding to the original graph. In particular, there exists in the Kron reduced graph of the nonlinear network an edge between boundary nodes $i$ and $k$ if and only if the $(i,k)$-th element of this Schur complement is non-zero. Furthermore, also in the nonlinear case the Kron reduced graph has an edge between boundary nodes $i$ and $k$ if and only if there is a path from $i$ to $k$ in the original graph. Finally, \eqref{Schur} can be obtained by successively eliminating the central nodes one by one \cite{vds10}.
}
\end{remark}
\begin{remark}
{\rm
With $z_B$ denoting the vector of \emph{inputs} of the network, the vector $\frac{\partial K}{\partial z}(z)$ is the natural vector of \emph{outputs} of the network. (For example, in a nonlinear resistor network these outputs are the \emph{nodal currents} at the boundary nodes, with $z_B$ denoting the boundary voltage potentials.) Similarly, the input-output map of the Kron reduced nonlinear network is given by $z_B \mapsto \frac{\partial \widehat{K}}{\partial z_B}(z_B) = \frac{\partial K}{\partial z_B}(z_C(z_B),z_B)$. 
}
\end{remark}
\begin{remark}
{\rm
In the linear case $K(z)= \frac{1}{2}z^\top D \bar{G} D^\top z$ for a diagonal matrix $\bar{G}$, and it follows \cite{bollobas} that $K(z)$ is independent of the orientation of the graph. However, for general $G(y)$ the function $K(z)=G(D^\top z)$ \emph{does} depend on the chosen orientation; as illustrated in Example \ref{example}.
}
\end{remark}

\section{Application to nonlinear resistor and memristor networks}
\label{sec3}
\subsection{Nonlinear resistor networks}
As already mentioned in Remark \ref{resis}, in the context of a nonlinear resistor network
the vector $z \in \mR^n$ is the vector of \emph{voltage potentials} $\psi$ associated with the nodes. Furthermore, $y=D^\top z$ is the vector of \emph{voltages} $V$ across the edges, and the current through the $j$-th edge is given by
\bq
I_j = \frac{\partial G_j}{\partial V_j}(V_j) .
\eq
In this context the functions $G_j$ are also referred to as \emph{co-content} functions; see \cite{millar,jeltsema}. Since the $G_j$ are assumed to be strongly convex, the nonlinear conductance functions $\frac{\partial G_j}{\partial V_j}(V_j)$ are strictly monotonically increasing. In particular, if additionally $\frac{\partial G_j}{\partial V_j}(0)=0$, then the power $V_j \frac{\partial G_j}{\partial V_j}(V_j)$ dissipated at every edge is greater than or equal to zero. Note that a \emph{linear} resistor $I_j= \bar{G}_jV_j$ with conductance $\bar{G}_j >0$ corresponds to a \emph{quadratic} function $G_j(V_j)=\frac{1}{2}\bar{G}_jV_j^2$.

Furthermore, $\frac{\partial K}{\partial z}(z)$ is the vector of \emph{nodal currents} $J$ at the nodes of the network. Hence $V^\top I=y^\top \frac{\partial G}{\partial y}(y)$ equals the total dissipated power at all the edges, and $\psi^\top J=  z^\top \frac{\partial K}{\partial z}(z)$ is the total power at the nodes. Since $V=D^\top \psi$ (Kirchhoff's voltage law) and $D I = J$ (Kirchhoff's current law), one obtains the power balance 
\bq
V^\top I= \psi^\top D I = \psi^\top J .
\eq
Kron reduction with respect to a splitting of the nodes into boundary and central nodes, and of the voltage potentials $z=\psi$ into $z=(z_C,z_B)=(\psi_C,\psi_B)$, corresponds to \emph{zero} nodal currents at the central nodes, that is
\bq
\frac{\partial K}{\partial \psi_C}(\psi_C,\psi_B)=J_C=0 .
\eq
If by Assumption \ref{ass1} this is solved for the central voltage potentials $\psi_C$ as a function of the boundary voltage potentials $\psi_B$, and furthermore Assumption \ref{ass2} is satisfied, then by Theorem \ref{theoremone} (acyclic case) or Theorem \ref{theoremtwo} (general case under additional condition) this results in a new nonlinear resistor network whose nodes are the boundary nodes of the original network. 

A special case of Kron reduction is to consider only \emph{two} boundary nodes, while all other nodes are central. Elimination of the central nodes (imposing that the nodal currents at the central nodes are zero), results in a Kron reduced nonlinear network with only two nodes, and a \emph{single} nonlinear conductance specified by $\widehat{G}(V)$, with $V$ the voltage across the edge between the two boundary nodes. (Note that this case is trivially covered by Theorem \ref{theoremone}.) This co-content function can be called the \emph{effective} co-content function, defining the \emph{effective conductance} $I= \frac{d \widehat{G}}{d V}(V)$.

In a \emph{linear} resistor network, the function $K(z)$ takes the quadratic form $K(z) = \frac{1}{2}z^\top D \bar{G} D^\top z$,
with $\bar{G}$ the diagonal matrix of positive conductances at the edges. In this case $K(z)$ is equal to the total \emph{dissipated power} in the network. Therefore, solving the constraint equations $0 = \frac{\partial K}{\partial z_C}(z_C,z_B)$ for $z_C$ means \emph{minimizing} the dissipated power as a function of $z_C$ for given $z_B$. This equivalence is called \emph{Maxwell's minimum heat theorem} (or Thomson's principle); see \cite{bollobas,millar}. This does \emph{not} hold anymore for nonlinear resistor networks, although $K(z)$ still has \emph{dimension} of power. In fact, for general $K$ the dissipated power is given as
\bq
\label{dissipated}
z^\top \frac{\partial K}{\partial z}(z) = z^\top D \frac{\partial G}{\partial y}(D^\top z) .
\eq
The minimization of this expression over $z_C$ (given $z_B$) is the same as solving $0 = \frac{\partial K}{\partial z_C}(z_C,z_B)$ over $z_C$ if
$z^\top \frac{\partial K}{\partial z}(z) = k K(z)$
for some positive $k$. By Euler's homogeneous function theorem this yields the following partial result.
\begin{proposition}
If $K$ is homogeneous of degree $k$ for some $k>0$, then solving $0 = \frac{\partial K}{\partial z_C}(z_C,z_B)$ for $z_C$ is the same as minimizing the dissipated power over $z_C$.
\end{proposition}
\begin{example}
\label{example}
As an application of Theorem \ref{theoremone}, consider the \emph{series interconnection} of two nonlinear resistors, both with conductance function
\bq
\label{char}
I = e^V - 1,
\eq
and thus co-content function $e^V-V$. This describes an idealized (and normalized) $I$-$V$ characteristic of a p-n diode; see e.g. \cite{chuadesoer}. The underlying graph is the line graph consisting of two boundary nodes with variables $z_1$ and $z_2$, and one central node with variable $z_0$. Since the characteristic \eqref{char} is \emph{not} odd, the orientation of the edges does matter. Choosing \emph{opposite} orientation for the two edges one obtains 
\bq
\label{KR}
K(z_0,z_1,z_2)= e^{z_0-z_1} - (z_0-z_1) + e^{z_0-z_2}  - (z_0-z_2).
\eq
Then Kron reduction leads to the constraint equation
\[
0=\frac{\partial K}{\partial z_0}(z_0,z_1,z_2)= e^{z_0} \left( e^{-z_1} + e^{-z_2} \right) -2,
\]
and thus $e^{z_0} \left( e^{-z_1} + e^{-z_2}\right) = 2, z_0= - \ln \left(e^{-z_1} + e^{-z_2} \right) + \ln 2$. Substitution into \eqref{KR} yields
\[
\widehat{K}(z_1,z_2)= 2 \ln \left(e^{-z_1} + e^{-z_2} \right) + z_1 + z_2 + 2 - 2 \ln 2.
\]
It is immediately checked that $\widehat{K}$ is shift invariant, and thus can be written as $\widehat{G}(z_1 -z_2)$ for some function $\widehat{G}$, which is computed as follows. Define $x_1:=z_1 - z_2, x_2:=z_1 + z_2$, and express $\widehat{K}$ in these new variables as
\[
\begin{array}{rcl}
\widehat{K} \!\! & =  & \! \! 2 \ln (e^{-\frac{1}{2}(x_1 + x_2)} + e^{-\frac{1}{2}(-x_1 + x_2)}) + x_2+ 2 - 2\ln 2 \\[2mm]
& = & \! \!  2 \ln \left[ e^{-\frac{1}{2}x_2} (e^{-\frac{1}{2}x_ 1} + e^{\frac{1}{2}x_1}) \right] + x_2 + 2 - 2 \ln 2 \\[2mm]
& = & \! \!  2 \ln (e^{-\frac{1}{2}x_ 1} + e^{\frac{1}{2}x_1}) + 2 - 2 \ln 2 .
\end{array}
\]
Denoting the voltage $V=z_1 -z_2=x_1$ one thus obtains
\[
\widehat{G}(V)= 2 \ln (e^{-\frac{1}{2}V} + e^{\frac{1}{2}V}) + 2 - 2 \ln 2,
\]
by differentiation leading to the \emph{effective conductance}
\bq
\label{I}
I = \frac{e^{\frac{1}{2}V} - e^{-\frac{1}{2}V}}{e^{\frac{1}{2}V} + e^{-\frac{1}{2}V}}\; 
= \tanh \frac{V}{2} 
\eq
of the equivalent nonlinear resistor. On the other hand, choosing instead the \emph{same} orientation for the two edges, a similar computation yields an equivalent resistor with effective conductance $I = e^{\frac{1}{2}V} -1$.
\end{example}

\subsection{Memristor networks}
An ideal memristor \cite{chua} is defined by a relation between \emph{charge} $q$ and \emph{flux} $\varphi$ such that the power $\dot{q} \cdot \dot{\varphi}$ is greater than or equal to zero. A flux-controlled memristor is defined by a function $q=M(\varphi)$ with $\frac{dM}{d \varphi}(\phi)\geq 0$, called the \emph{memductance} function. Throughout we will assume the slightly stronger requirement that $M$ is strongly convex.

In the context of a \emph{network} of flux-controlled memristors, briefly \emph{memristor network}, $z$ is the vector of \emph{nodal fluxes} associated with the nodes, and $y=D^\top z$ is the vector $\varphi$ of fluxes across the edges. Furthermore, $\frac{\partial G}{\partial y}(y)$ is the vector $q$ of charges at the edges (with the function $G_j$ specifying the $j$-the memristor), while $\frac{\partial K}{\partial z}(z)$ is the vector of \emph{nodal} charges. The function $K(z)= G(D^\top z)$ also appeared in \cite{chua}, and was called there the \emph{action}. 

Kron reduction of a memristor network amounts to setting the charges at the central nodes of the graph equal to zero. Application of the theory of the preceding section shows how this results in a Kron reduced memristor network that only consists of the boundary nodes, with strongly convex flux-controlled memristors at the $\hat{m}$ new edges of the Kron reduced graph. By taking as boundary nodes any pair of nodes of a memristor network, Kron reduction leads to a single memristor between these two nodes, whose memductance function is the \emph{effective} memductance between the two selected nodes.

\section{Outlook}
\label{sec5}
From a mathematical point of view the main reasoning of this paper is as follows. Given the function $K(z)=G(D^\top z)$ associated to the nonlinear network, the function $\widehat{K}(z_B)$ obtained from $K(z)$ by solving for $z_C$ from $0=\frac{\partial K}{\partial z_C}(z_C,z_B)$ is, under appropriate assumptions, of the same form $\widehat{K}(z_B)=\widehat{G}(\widehat{D}^\top z_B)$. Here $\widehat{D}$ is the incidence matrix of the Kron reduced graph and $\widehat{G}$ is a diagonal mapping whose elements define the 'conductances' of the new edges. Somewhat surprisingly, this is approached by considering the Hessian matrix of $K$, which for every $z$ is a Laplacian matrix, and by using the result of \cite{vds10} that any Schur complement of a Laplacian matrix is again a Laplacian matrix. 

Assumptions \ref{ass1} and \ref{ass2} turned out to be instrumental in deriving the main results. 
Assumption \ref{ass1} can be considered as a regularity assumption, but an explicitly checkable condition for Assumption \ref{ass2} would be desirable. While satisfaction of Assumptions \ref{ass1} and \ref{ass2} is sufficient in case the Kron reduced graph does not contain cycles (Theorem \ref{theoremone}), the general case (Theorem \ref{theoremtwo}) additionally requires the integrability conditions \eqref{integ}. Again, obtaining more easily checkable conditions for satisfaction of \eqref{integ} would be desirable.

Overall, the computation of the mapping $\widehat{G}$ is much more demanding than in the linear case. More insight into this would provide valuable information about the properties of the Kron reduced nonlinear network, and in particular about the effective conductance and effective memductance of Kron reduced nonlinear resistor, respectively, memristor, networks. Kron reduction of other examples of nonlinear networks, see e.g. the discussion in \cite{vds,vanderschaftmaschkeBosgra}, can be studied in more detail as well. Another venue for research concerns the \emph{synthesis} problem of nonlinear networks; see already \cite{vds10} for the linear case.



\section*{Appendix}
Let us summarize Theorem 3.1 in \cite{vds10}. Consider as before a directed graph with $n \times m$ incidence matrix $D$. Let $W$ be a positive definite $m \times m$ diagonal matrix. The matrix $DWD^\top$ is called a \emph{Laplacian matrix} (with weight matrix $W$). 
\begin{theorem}\label{theorem1}
\begin{enumerate}
\item
The Laplacian $DWD^\top$ is symmetric, positive semi-definite, and independent of the orientation of the graph. Furthermore, it has all diagonal elements $\geq 0$, all off-diagonal elements $\leq 0$, and has zero row and column sums. The vector $\mathds{1}$ is in the kernel of $DWD^\top$, and if the graph is connected then $\ker DWD^\top = \spa \mathds{1}$. If the graph is connected, then all diagonal elements of $DWD^\top$ are strictly positive.
\item
Every positive semi-definite matrix $L$ with diagonal elements $\geq 0$, off-diagonal elements $\leq 0$, and with zero row and column sums can be written as the Laplacian matrix $L=DWD^\top$, with $D$ an incidence matrix, and $W$ a positive definite diagonal matrix.
\item
If the graph is connected, then all Schur complements of $DWD^\top$ are well-defined, positive semi-definite, with diagonal elements $>0$, off-diagonal elements $\leq 0$, and with zero row and column sums. In particular, all Schur complements of $DWD^\top$ can be written as $\widehat{D}\widehat{W}\widehat{D}^\top$, with $\widehat{D}$ the incidence matrix of a connected graph, and $\widehat{W}$ a positive definite diagonal matrix. 
\end{enumerate}
\end{theorem}

\begin{lemma}\label{lemmaA}
If $K(z_C,z_B)$ is strongly convex, then $\widehat{K}(z_B)= K(z_C(z_B),z_B)$, with $z_C(z_B)$ determined by $\frac{\partial K}{\partial z_C}(z_C,z_B)=0$, is also strongly convex.
\end{lemma}
\begin{proof}
Strong convexity is equivalent to the Hessian matrix being positive definite. Thus, since $K(z)$ is strongly convex its Hessian matrix is positive definite. On the other hand, in view of \eqref{Schur}, the Hessian matrix of $\widehat{K}(z_B)$ is a Schur complement of the Hessian matrix of $K(z)$. Since the Schur complement of a positive definite matrix is again positive definite, cf. \cite{crab}, the result follows.
\end{proof}
\begin{lemma}\label{lemmaB}
$\widehat{K}$ is invariant under shifts $z_B \mapsto z_B + c \mathds{1}$. 
\end{lemma}
\begin{proof}
For each $c \in \mR$
\[
\begin{array}{rcl}
K(z_C(z_B),z_B) & \leq & K(z_C(z_B + c\mathds{1}) -c \mathds{1},z_B) \\[2mm]
& = & K(z_C(z_B +c \mathds{1}),z_B +c \mathds{1}) \\[2mm]
& \leq & K(z_C(z_B) +c \mathds{1},z_B +c \mathds{1}) \\[2mm]
&= & K(z_C(z_B),z_B) .
\end{array}
\]
The two inequalities follow from the fact that $z_C(z_B)$ is minimizing $K(z_C,z_B)$ for each $z_B$, while the two equalities follow from shift invariance of $K$. Hence $\widehat{K}(z_B +c\mathds{1})= \widehat{K}(z_B)$, as claimed. (Note that additionally this shows that $z_C(z_B + c \mathds{1}) = z_C(z_B) + c \mathds{1}$.)
\end{proof}

\section*{Acknowledgement}
The first author thanks Nima Monshizadeh for a discussion leading to the formulation of Lemma \ref{lemma}.

\end{document}